\documentclass[11pt]{article}
\usepackage{amssymb}
\usepackage{amsfonts}
\usepackage{mitpress}

\newtheorem{theorem}{Theorem}

\newtheorem{corollary}[theorem]{Corollary}

\newtheorem{definition}[theorem]{Definition}

\newtheorem{proposition}[theorem]{Proposition}
\newtheorem{remark}[theorem]{Remark}

\newenvironment{proof}[1][Proof]{\noindent\textbf{#1.} }{\ \rule{0.5em}{0.5em}}
\newdimen\dummy
\dummy=\oddsidemargin
\input{tcilatex}

\begin{document}

\title{Some variational properties of the weighted $\sigma _{r}-$curvature
for submanifolds in Riemannian manifolds}
\author{Mohammed Benalili \\
Dept. Maths, Facult\'{e} des Sciences, Universit\'{e} UABBT Tlemcen Alg\'{e}%
rie.\\
m\_benalili@yahoo.fr}
\maketitle

\begin{abstract}
The objet of this paper is the study of the variations of a functional whose
integrant is the $r$-th weighted curvature on the hypersurface of a closed
Riemannian manifold. Some applications to hypersurfaces of the Euclidean
space and the unit round sph\`{e}re are given. .
\end{abstract}

\section{\ Introduction}

The study of Riemannian geometry seems to be based essentially on the study
of certain operators such as the shape operator, the Ricci tensor, the
Schouten operator etc...Some functions constructed from these operators play
a fundamental role in Riemannian geometry. Particularly the algebraic
invariants of these operators such as the $r$-th symmetric functions $\sigma
_{r}$ associated with the shape operator and the Newton transformations $%
T_{r}$.The following articles can be consulted on this subject (\cite{M.A1}, 
\cite{M.A2}, \cite{Case}, \cite{Cis}, \cite{K.N1}, \cite{K.N2}, \cite{Reilly}%
, \cite{Reilly2}, \cite{Rund}). Reilly ( see \cite{Reilly} ) has considered
the variations where the integrant of the functional is a function of the $r$%
-th mean curvatures $\sigma _{r}$. In a recent paper Case (see \cite{Case})
introduced and studied the notion of \ $r$-th weighted curvatures. The aim
of this paper is the study of the variations of a functional whose integrant
is $r$-th weighted curvature on the hypersurface of a closed Riemannian
manifold. Some applications to hypersurfaces of the Euclidean space and the
unit round sphere are given.

\section{Preliminaries}

We associate with any endomorphism $A$ on an $n-$dimensional vector space $V$
a family of endomorphisms $\left( T_{r}\right) _{r}$ on $V$ defined
recurrently by:

\[
T_{0}=id_{V} 
\]

\[
T_{r}=\sigma _{r}id_{V}-AT_{r-1},r=1,2,... 
\]

or under a condensed formula%
\[
T_{r}=\dsum\limits_{j=0}^{r}\left( -1\right) ^{j}\sigma _{r-j}A^{j} 
\]%
with $T_{r}=0$, for all $r\geq n.$

First we recall ( see \cite{Case} ) the notion of the weighted elementary
symmetric polynomials.

Let $r\in N\cup \left\{ \infty \right\} $. The $r$-th weighted elementary
symmetric polynomial $\sigma _{r}^{\infty }:R\times R^{n}\rightarrow R,$ is
inductively given by 
\begin{equation}
\left\{ 
\begin{array}{c}
\sigma _{0}^{\infty }(\mu _{o},\mu )=1\ \ \ \ \ \ \ \ \ \ \ \ \ \ \ \ \ \ \
\ \ \ \ \ \ \ \ \ \ \ \ \ \ \ \ \ \ \ \ \ \ \ \ \ \ \ \ \ \ \ \ \ \ \ \ \ \
\ \ \ \ \ \ \ \ \ \ \ \ \ \ \ \ \ \ \ \ \ \ \ \ \ \ \ \ \ \ \  \\ 
\ \ \sigma _{r}^{\infty }(\mu _{o},\mu )=\sigma _{r-1}^{\infty }(\mu
_{o},\mu )\dsum\limits_{j=1}^{n}\mu
_{j}+\dsum\limits_{i=1}^{r-1}\dsum\limits_{j=1}^{n}\left( -1\right)
^{i}\sigma _{r-i-1}^{\infty }(\mu _{o},\mu )\mu _{j}^{i}\text{, for }%
r\geqslant 1\text{ \ \ \ \ \ \ \ \ \ \ \ \ \ \ \ \ \ \ } \\ 
\text{and }\mu =\left( \mu _{1},...,\mu _{n}\right) \in R^{n}\text{\ .\ \ \
\ \ \ \ \ \ \ \ \ \ \ \ \ \ \ \ \ \ \ \ \ \ \ \ \ \ \ \ \ \ \ \ \ \ \ \ \ \
\ \ \ \ \ \ \ \ \ \ \ \ \ \ \ \ \ \ \ \ \ \ \ \ \ \ \ \ \ \ \ }%
\end{array}%
\text{ }\right.  \tag{1}  \label{1}
\end{equation}%
$\sigma _{r}^{\infty }(\mu _{o},\mu )$ is expressed in terms of $\mu _{o}$
and $1,\sigma _{1}(\mu ),...,\sigma _{r}(\mu )$ ( see \cite{Case} page 6):

\begin{proposition}
Given $r\in N\cup \left\{ \infty \right\} $, $\mu _{o}\in R$ and $\mu \in
R^{n}$, we have 
\begin{equation}
\sigma _{r}^{\infty }(\mu _{o},\mu )=\dsum\limits_{j=0}^{r}\frac{\mu _{0}^{j}%
}{j!}\sigma _{r-j}(\mu ).  \tag{2}  \label{2}
\end{equation}
\end{proposition}

\begin{definition}
Since a symmetric matrix is {}{}diagonalizable one can consider the
symmetric elementary polynomials associated with a real number $\mu _{o}$
and a symmetric matrix $A$ with eigenvalues $\mu =\left( \mu _{1},...,\mu
_{n}\right) $ like being the polynomials acting on that real number $\mu
_{o} $ and on the vector $\mu =\left( \mu _{1},...,\mu _{n}\right) $, so the
weighted elementary symmetric polynomial of $A$ is defined by $\sigma
_{r}^{\infty }(\mu _{o},A)=$ $\sigma _{r}^{\infty }(\mu _{o},\mu ).$
\end{definition}

The weighted Newton transformations are then defined as follows ( see \cite%
{Case})

Let $A$ be a symmetric $n\times n$ real-valued matrix and $\mu _{o}\in R$.
The $r$-th weighted Newton transformation function $T_{r}^{\infty }(\mu
_{o},A)$ of $A$ and $\mu _{o}$ is 
\begin{equation}
T_{r}^{\infty }(\mu _{o},A)=\dsum\limits_{j=0}^{r}\left( -1\right)
^{j}\sigma _{r-j}^{\infty }(\mu _{o},A)A^{j}\text{.}  \tag{3}  \label{3}
\end{equation}%
The weighted elementary symmetric polynomials can be expressed in terms of
the weighted Newton transformations in the following way ( see \cite{Case})

\begin{proposition}
\label{prop3} Let $A$ be a symmetric $n\times n$ real-valued matrix and $\mu
_{o}\in R$. It follows that 
\begin{equation}
\text{trace}(AT_{r}^{\infty }(\mu _{o},A))=\left( r+1\right) \sigma
_{r+1}^{\infty }(\mu _{o},A)-\mu _{o}\sigma _{r}^{\infty }(\mu _{o},A). 
\tag{5}  \label{5}
\end{equation}
\end{proposition}

\begin{proposition}
\begin{equation}
T_{r}^{\infty }(\mu _{o},A)=\dsum\limits_{j=0}^{r}\frac{\mu _{o}^{j}}{j!}%
T_{r-j}\left( A\right)  \tag{6}  \label{6}
\end{equation}%
where $T_{r}\left( A\right) $ is the classical Newton transformation.
\end{proposition}

\begin{proof}
From the formulae (\ref{2}) and (\ref{3}), we get%
\begin{eqnarray*}
T_{r}^{\infty }(\mu _{o},A) &=&\dsum\limits_{l=0}^{r}\left( -1\right)
^{j}\sigma _{r-l}^{\infty }(\mu _{o},A)A^{l} \\
&=&\dsum\limits_{l=0}^{r}\left( -1\right) ^{l}\dsum\limits_{j=0}^{r-l}\frac{%
\mu _{o}^{j}}{j!}\sigma _{r-l-j}(A)A^{l} \\
&=&\dsum\limits_{l=0}^{r}\frac{\mu _{o}^{l}}{l!}T_{r-l}\left( A\right) .
\end{eqnarray*}
\end{proof}

As it is well known that the classical Newton transformations are of
free-divergence i.e. 
\[
T_{r,j}^{ij}\left( A\right) =0 
\]%
we deduce that:

\begin{proposition}
\label{prop4}%
\[
T_{r}^{\infty }(\mu _{o},A)_{i,j}^{j}=\dsum\limits_{l=0}^{r-1}\frac{\mu
_{o}^{l}}{l!}\left( T_{r-1-l}\right) _{i}^{j}\left( A\right) \mu _{o,j}\text{%
.} 
\]
\end{proposition}

From the formula (\ref{2}) i.e.%
\[
\sigma _{r+1}^{\infty }(\mu _{o},A)=\sum_{j=0}^{r+1}\frac{\mu _{o}^{j}}{j!}%
\sigma _{r+1-j}(A) 
\]%
we obtain by differentiation with respect to $t$ that%
\begin{eqnarray*}
\frac{\partial }{\partial t}\sigma _{r+1}^{\infty }(\mu _{o},A)
&=&\sum_{j=0}^{r+1}\frac{\mu _{o}^{j}}{j!}\frac{\partial }{\partial t}\sigma
_{r+1-j}+\sum_{j=1}^{r+1}\frac{\mu _{o}^{j-1}}{\left( j-1\right) !}\sigma
_{r+1-j}\frac{\partial \mu }{\partial t} \\
&=&\sum_{j=0}^{r+1}\frac{\mu _{o}^{j}}{j!}\frac{\partial }{\partial t}\sigma
_{r+1-j}+\sum_{j=0}^{r}\frac{\mu _{o}^{j}}{j!}\sigma _{r-j}\frac{\partial
\mu }{\partial t}.
\end{eqnarray*}%
Under the following formula (see \cite{Reilly} Lemma A page 467)%
\begin{equation}
\frac{\partial \sigma _{r+1}}{\partial t}=\text{trace}\left( \frac{\partial A%
}{\partial t}T_{r}\right)  \tag{7}  \label{7}
\end{equation}%
and the Newton's formula ( see \cite{Reilly} formula (1) page 467 ) 
\begin{equation}
\left( r+1\right) \sigma _{r+1}=\text{trace}(AT_{r}),\text{ }  \tag{8}
\label{8}
\end{equation}%
we get%
\begin{eqnarray}
\frac{\partial }{\partial t}\sigma _{r+1}^{\infty }(\mu _{o},A)
&=&\sum_{j=0}^{r+1}\frac{\mu _{o}^{j}}{j!}\text{trace}\left( \frac{\partial A%
}{\partial t}T_{r-j}\right) +\sum_{j=0}^{r}\frac{\mu _{o}^{j}}{j!}\sigma
_{r-j}\frac{\partial \mu _{o}}{\partial t}  \nonumber \\
&=&\sum_{j=0}^{r+1}\frac{\mu _{o}^{j}}{j!}\text{trace}\left( \frac{\partial A%
}{\partial t}T_{r-j}\right) +\sigma _{r}^{\infty }\left( \mu _{o},A\right) 
\frac{\partial \mu _{o}}{\partial t}  \TCItag{9}  \label{9} \\
&=&\text{trace}\left( \frac{\partial A}{\partial t}T_{r}^{\infty }(\mu
_{o},A)\right) +\sigma _{r}^{\infty }\left( \mu _{o},A\right) \frac{\partial
\mu _{o}}{\partial t}.  \nonumber
\end{eqnarray}%
\ \ Consider a one family of parameter $\ \psi _{t}:M^{m}\rightarrow 
\overline{M}^{n}$ of immersions of an $m-$dimensional closed manifold $M^{m}$
into an $n$-Riemannian manifold $\left( \overline{M}^{n}\text{,}\left\langle
,\right\rangle \right) $. Denote by $X$ the deformation vector field and by $%
\nu $ the normal vector field to $\overline{M}^{n}$. Put $\lambda
=\left\langle X,\nu \right\rangle $, $\mu =X^{\top }$ the tangential
component of $X$ and $dV$ the volume form on $M^{m}$ . Consider the
following variational problem%
\[
\delta \left( \int_{M}\sigma _{k}^{\infty }dV\right) =0 
\]

\begin{theorem}
\label{thm1} With the above notations and assumptions the first variation of
the global $\sigma _{r}^{\infty }-$curvature is given by: 
\[
\frac{d}{dt}\int_{M}\sigma _{r}^{\infty }dV=\int_{M}\left\{ \lambda \left(
-\left( r+1\right) \sigma _{r+1}^{\infty }+\mu _{o}\sigma _{r}^{\infty
}-\sigma _{r-1}^{\infty }\mu ^{l}\mu _{o,l}\right) +\lambda ,_{ij}\left(
T_{r}^{\infty }\right) ^{ij}\right. 
\]%
\[
-g^{jm}\left( R^{\overline{M}}\right) \left( \nu ,\frac{\partial \psi }{%
\partial x_{i}},X,\frac{\partial \psi }{\partial x_{m}}\right) \left(
T_{r-1}^{\infty }\right) _{j}^{i}
\]%
\[
\left. +g^{im}\left( R^{\overline{M}}\right) ^{\bot }\left( \frac{\partial
\psi }{\partial x_{j}},\mu \right) \frac{\partial \psi }{\partial x_{m}}%
\left( T_{r-1}^{\infty }\right) ^{ij}+\sigma _{r}^{\infty }\left( \mu
_{o},A\right) \frac{\partial \mu _{o}}{\partial t}\right\} dV.
\]
\end{theorem}

On the other hand%
\[
\text{trace}\left( \frac{\partial A}{\partial t}T_{k}^{\infty }\right) =%
\frac{\partial A_{j}^{i}}{\partial t}\left( T_{k}^{\infty }\right) _{i}^{j} 
\]%
with%
\[
A_{i}^{j}=g^{jk}A_{ik} 
\]%
where 
\[
A_{ik}=\left\langle \overline{\nabla }_{\frac{\partial }{\partial x_{i}}}%
\frac{\partial \psi }{\partial x_{k}}),\nu \right\rangle =-\left\langle 
\overline{\nabla }_{\frac{\partial \psi }{\partial x_{i}}}\nu ,\frac{%
\partial \psi }{\partial x_{k}}\right\rangle 
\]%
Hence 
\[
\frac{\partial A_{i}^{j}}{\partial t}=\frac{\partial g^{jk}}{\partial t}%
A_{ik}+g^{jk}\frac{\partial A_{ik}}{\partial t}\text{.} 
\]%
Obviously%
\[
\frac{\partial g^{jk}}{\partial t}=-g^{jl}\frac{\partial g_{pl}}{\partial t}%
g^{pk} 
\]%
Now, if we consider the calculations in a normal coordinates that is at a
point $x\in M$ where the metric tensor fulfills $g_{ij}(x)=\left\langle 
\frac{\partial \psi }{\partial x_{i}},\frac{\partial \psi }{\partial x_{j}}%
\right\rangle =\delta _{ij}$ and $\Gamma _{ij}^{k}(x)=0$, where $\Gamma
_{ij}^{k}$ stand for the Christoffel symbols corresponding to the metric
connection $\nabla $ on $M$, we get%
\begin{eqnarray*}
\frac{\partial g_{pl}}{\partial t} &=&\left\langle \overline{\nabla }_{\frac{%
\partial }{\partial t}}\frac{\partial \psi }{\partial x_{p}},\frac{\partial
\psi }{\partial x_{l}}\right\rangle +\left\langle \frac{\partial \psi }{%
\partial x_{p}},\overline{\nabla }_{\frac{\partial }{\partial t}}\frac{%
\partial \psi }{\partial x_{l}}\right\rangle \\
&=&\left\langle \overline{\nabla }_{\frac{\partial }{\partial x_{p}}}X,\frac{%
\partial \psi }{\partial x_{l}}\right\rangle +\left\langle \frac{\partial
\psi }{\partial x_{p}},\overline{\nabla }_{\frac{\partial }{\partial x_{l}}%
}X\right\rangle \\
&=&\left\langle \overline{\nabla }_{\frac{\partial }{\partial x_{p}}}\left(
\lambda \nu +\mu ^{m}\frac{\partial \psi }{\partial x_{m}}\right) ,\frac{%
\partial \psi }{\partial x_{l}}\right\rangle +\left\langle \frac{\partial
\psi }{\partial x_{p}},\overline{\nabla }_{\frac{\partial }{\partial x_{l}}%
}\left( \lambda \nu +\mu ^{m}\frac{\partial \psi }{\partial x_{m}}\right)
\right\rangle \\
&=&\mu _{,l}^{p}+\mu _{,p}^{l}-2\lambda A_{pl}.
\end{eqnarray*}%
hence%
\begin{equation}
\frac{\partial g^{jk}}{\partial t}=-g^{jl}g^{pk}\left( \mu _{p,l}+\mu
_{l,p}-2\lambda A_{pl}\right) .  \tag{10}  \label{10}
\end{equation}

We have also%
\begin{eqnarray}
\frac{\partial \nu }{\partial t} &=&\left\langle \frac{\partial \nu }{%
\partial t},\frac{\partial \psi }{\partial x_{k}}\right\rangle \frac{%
\partial \psi }{\partial x_{k}}  \TCItag{11}  \label{11} \\
&=&-\left\langle \nu ,\overline{\nabla }_{\frac{\partial }{\partial t}}\frac{%
\partial \psi }{\partial x_{k}}\right\rangle \frac{\partial \psi }{\partial
x_{k}}  \nonumber \\
&=&-\left\langle \nu ,\overline{\nabla }_{\frac{\partial }{\partial x_{k}}%
}X\right\rangle \frac{\partial \psi }{\partial x_{k}}  \nonumber \\
&=&-h^{jk}\left( \lambda ,_{j}+\mu ^{l}A_{jl}\right) \frac{\partial \psi }{%
\partial x_{k}}.  \nonumber
\end{eqnarray}%
Now we compute%
\begin{eqnarray*}
\frac{\partial A_{ik}}{\partial t} &=&-\left\langle \overline{\nabla }_{%
\frac{\partial }{\partial t}}\overline{\nabla }_{\frac{\partial }{\partial
x_{i}}}\nu ,\frac{\partial \psi }{\partial x_{k}}\right\rangle -\left\langle 
\overline{\nabla }_{\frac{\partial }{\partial x_{i}}}\nu ,\overline{\nabla }%
_{\frac{\partial }{\partial t}}\frac{\partial \psi }{\partial x_{k}}%
\right\rangle  \\
&=&-R^{\overline{M}}(\nu ,\frac{\partial \psi }{\partial x_{k}},\frac{%
\partial \psi }{\partial x_{i}},X)-\left\langle \overline{\nabla }_{\frac{%
\partial }{\partial x_{i}}}\overline{\nabla }_{\frac{\partial }{\partial
x_{t}}}\nu ,\frac{\partial \psi }{\partial x_{k}}\right\rangle -\left\langle 
\overline{\nabla }_{\frac{\partial }{\partial x_{i}}}\nu ,\overline{\nabla }%
_{\frac{\partial }{\partial x_{k}}}X\right\rangle 
\end{eqnarray*}%
By formula (\ref{10}), we get%
\[
\overline{\nabla }_{\frac{\partial }{\partial x_{i}}}\overline{\nabla }_{%
\frac{\partial }{\partial x_{t}}}\nu =-g^{jm}\left( \left( \lambda
,_{ji}+\mu ^{l,i}A_{jl}+\mu ^{l}A_{jl_{,i}}\right) \frac{\partial \psi }{%
\partial x_{m}}+\left( \lambda ,_{j}+\mu ^{l}A_{jl}\right) \overline{\nabla }%
_{\frac{\partial }{x_{i}}}\frac{\partial \psi }{\partial x_{m}}\right) 
\]%
so%
\[
\left\langle \overline{\nabla }_{\frac{\partial }{\partial x_{i}}}\overline{%
\nabla }_{\frac{\partial }{\partial x_{t}}}\nu ,\frac{\partial \psi }{%
\partial x_{k}}\right\rangle =-g^{jm}\left( \lambda ,_{ji}+\mu
^{l,i}A_{jl}+\mu ^{l}A_{jl_{,i}}\right) g_{mk}
\]%
\[
-\left( \lambda ,_{j}+\mu ^{l}A_{jl}\right) g^{jm}\left\langle \overline{%
\nabla }_{\frac{\partial }{x_{i}}}\frac{\partial \psi }{\partial x_{m}},%
\frac{\partial \psi }{\partial x_{k}}\right\rangle .
\]%
In the same manner, we have 
\begin{eqnarray*}
\left\langle \overline{\nabla }_{\frac{\partial }{\partial x_{i}}}\nu ,%
\overline{\nabla }_{\frac{\partial }{\partial x_{k}}}X\right\rangle 
&=&\left\langle \overline{\nabla }_{\frac{\partial }{\partial x_{i}}}\nu ,%
\overline{\nabla }_{\frac{\partial }{\partial x_{k}}}\left( \lambda \nu +\mu
^{m}\frac{\partial \psi }{\partial x_{m}}\right) \right\rangle  \\
&=&\left\langle \overline{\nabla }_{\frac{\partial }{\partial x_{i}}}\nu
,\lambda \overline{\nabla }_{\frac{\partial }{\partial x_{k}}}\nu +\mu ^{m,k}%
\frac{\partial \psi }{\partial x_{m}}+\mu ^{m}\overline{\nabla }_{\frac{%
\partial }{\partial x_{k}}}\frac{\partial \psi }{\partial x_{m}}%
\right\rangle  \\
&=&\lambda \left\langle \overline{\nabla }_{\frac{\partial }{\partial x_{i}}%
}\nu ,\overline{\nabla }_{\frac{\partial }{\partial x_{k}}}\nu \right\rangle
-\mu ^{m,k}A_{im}+\mu ^{m}\left\langle \overline{\nabla }_{\frac{\partial }{%
\partial x_{i}}}\nu ,\overline{\nabla }_{\frac{\partial }{\partial x_{k}}}%
\frac{\partial \psi }{\partial x_{m}}\right\rangle .
\end{eqnarray*}%
By noticing that%
\[
\left\langle \overline{\nabla }_{\frac{\partial }{\partial x_{i}}}\nu ,%
\overline{\nabla }_{\frac{\partial }{\partial x_{k}}}\nu \right\rangle
=\left\langle \overline{\nabla }_{\frac{\partial }{\partial x_{i}}}\nu ,%
\frac{\partial \psi }{\partial x_{j}}\right\rangle \left\langle \overline{%
\nabla }_{\frac{\partial }{\partial x_{k}}}\nu ,\frac{\partial \psi }{%
\partial x_{j}}\right\rangle =A_{ik}^{2}
\]%
we get%
\begin{eqnarray*}
\frac{\partial A_{ik}}{\partial t} &=&-R^{\overline{M}}(\nu ,\frac{\partial
\psi }{\partial x_{k}},\frac{\partial \psi }{\partial x_{i}},X)+g^{jm}\left(
\lambda ,_{ji}+\mu ^{l,i}A_{jl}+\mu ^{l}A_{jl_{,i}}\right) g_{mk} \\
&&-\lambda A_{ik}^{2}+\mu ^{l,k}A_{il}.
\end{eqnarray*}%
Hence 
\begin{eqnarray*}
g^{jk}\frac{\partial A_{ik}}{\partial t} &=&-g^{jk}R^{\overline{M}}(\nu ,%
\frac{\partial \psi }{\partial x_{k}},\frac{\partial \psi }{\partial x_{i}}%
,X)+g^{jk}g^{pm}\left( \lambda ,_{pi}+\mu ^{l,i}A_{pl}+\mu
^{l}A_{pl,i}\right) g_{mk} \\
&&-\lambda g^{jk}A_{ik}^{2}+g^{jk}\mu ^{l,k}A_{il}
\end{eqnarray*}%
Taking into account formula (\ref{10}), we get%
\begin{eqnarray}
\frac{\partial A_{i}^{j}}{\partial t} &=&-g^{jk}R^{\overline{M}}(\nu ,\frac{%
\partial \psi }{\partial x_{k}},\frac{\partial \psi }{\partial x_{i}}%
,X)+g^{jk}\left( \lambda ,_{ki}+\mu ^{l,i}A_{kl}+\mu ^{l}A_{kl,i}\right)  
\TCItag{12}  \label{12} \\
&&-\lambda g^{jk}A_{ik}^{2}+g^{jk}\mu ^{l,k}A_{il}-g^{jl}g^{pk}\left( \mu
_{,p}^{l}+\mu _{,l}^{p}-2\lambda A_{pl}\right) A_{ik}.  \nonumber
\end{eqnarray}%
To compute $\frac{\partial \sigma _{r+1}^{\infty }}{\partial t}$ we multiply
both sides of (\ref{12}) by $\left( T_{r}^{\infty }\right) _{j}^{i}$ and sum
and add the term $\sigma _{r}^{\infty }\left( \mu _{o},A\right) \frac{%
\partial \mu _{o}}{\partial t}$.

First, we have

\begin{eqnarray*}
\lambda A_{p}^{j}A_{i}^{p}\left( T_{r}^{\infty }\right) _{j}^{i} &=&\lambda 
\text{trace}\left( A^{2}T_{r}^{\infty }\right) \\
&=&\lambda \text{trace}\left( \sigma _{r+1}^{\infty }A-AT_{r+1}^{\infty
}\right)
\end{eqnarray*}%
and by the formulas (\ref{2}) and (\ref{9}), we get%
\begin{equation}
\lambda A_{p}^{j}A_{i}^{p}\left( T_{r}^{\infty }\right) _{j}^{i}=\lambda
\left( \sigma _{1}^{\infty }\sigma _{r+1}^{\infty }-\left( r+2\right) \sigma
_{r+2}^{\infty }\right) .  \tag{13}  \label{13}
\end{equation}%
We also write%
\begin{equation}
g^{jk}\lambda ,_{ki}\left( T_{r}^{\infty }\right) _{j}^{i}=\left(
T_{r}^{\infty }\right) ^{ij}\lambda ,_{ij}.  \tag{14}  \label{14}
\end{equation}%
By the Codazzi formula,we have 
\[
g^{jm}\mu ^{l}A_{ml},_{i}\left( T_{r}^{\infty }\right) _{j}^{i}=g^{jm}\mu
^{l}A_{mi},_{l}\left( T_{r}^{\infty }\right) _{j}^{i}+g^{jm}\left( R^{%
\overline{M}}\right) ^{\bot }(\frac{\partial \psi }{\partial x_{i}},\mu )%
\frac{\partial \psi }{\partial x_{m}}\left( T_{r}^{\infty }\right) _{j}^{i} 
\]

By formula (\ref{9}) we get%
\begin{equation}
g^{jm}\mu ^{l}A_{ml},_{i}\left( T_{r}^{\infty }\right) _{j}^{i}=\mu
^{l}\sigma _{r+1,l}^{\infty }-\sigma _{r}^{\infty }\left( \mu _{o},A\right)
\mu ^{l}\mu _{o,l}+g^{jm}\left( R^{\overline{M}\ }\right) ^{\bot }(\frac{%
\partial \psi }{\partial x_{i}},\mu )\frac{\partial \psi }{\partial x_{m}}%
\left( T_{k}^{\infty }\right) _{j}^{i}.  \tag{15}  \label{15}
\end{equation}

Also, we have%
\begin{equation}
g^{jl}g^{pm}\left( \mu _{,p}^{l}+\mu _{,l}^{p}\right) A_{im}=2\mu
^{j,p}A_{ip}  \tag{16}  \label{16}
\end{equation}%
and%
\begin{equation}
g^{jk}\left( \mu ^{l,i}A_{kl}+\mu _{,k}^{l}A_{il}\right) =2\mu ^{l,i}A_{jl}%
\text{.}  \tag{17}  \label{17}
\end{equation}%
Hence%
\begin{eqnarray}
\frac{\partial A_{i}^{j}}{\partial t} &=&-g^{jk}R^{\overline{M}}(\nu ,\frac{%
\partial \psi }{\partial x_{k}},\frac{\partial \psi }{\partial x_{i}}%
,X)+g^{jk}\lambda ,_{ki}+g^{jk}\mu ^{l}A_{kl},i  \nonumber \\
&&-\lambda g^{jk}\left\langle \overline{\nabla }_{\frac{\partial }{\partial
x_{i}}}\nu ,\overline{\nabla }_{\frac{\partial }{\partial x_{k}}}\nu
\right\rangle +2\lambda g^{jl}g^{pk}A_{pl}A_{ik}.  \TCItag{18}  \label{18}
\end{eqnarray}%
and since%
\[
\lambda \left\langle \overline{\nabla }_{\frac{\partial }{\partial x_{i}}%
}\nu ,\overline{\nabla }_{\frac{\partial }{\partial x_{k}}}\nu \right\rangle
=\lambda \left\langle \overline{\nabla }_{\frac{\partial }{\partial x_{i}}%
}\nu ,\overline{\nabla }_{\frac{\partial }{\partial x_{k}}}\nu \right\rangle
=\lambda A_{ik}^{\left( 2\right) }
\]%
we infer that 
\begin{eqnarray}
\frac{\partial A_{i}^{j}}{\partial t} &=&-g^{jk}R^{\overline{M}}(\nu ,\frac{%
\partial \psi }{\partial x_{k}},\frac{\partial \psi }{\partial x_{i}}%
,X)+g^{jk}\lambda ,_{ki}+g^{jk}\mu ^{l}A_{kl},i  \nonumber \\
&&+\lambda g^{jl}g^{pk}A_{pl}A_{ik}.  \TCItag{19}  \label{19}
\end{eqnarray}%
Hence, by virtue of the formula (\ref{9}), we deduce%
\begin{eqnarray}
\frac{\partial }{\partial t}\sigma _{r+1}^{\infty } &=&-g^{jm}R^{\overline{M}%
}(\nu ,\frac{\partial \psi }{\partial x_{m}},\frac{\partial \psi }{\partial
x_{i}},X)\left( T_{r}^{\infty }\right) _{j}^{i}+\lambda ,_{ij}\left(
T_{r}^{\infty }\right) ^{ij}  \nonumber \\
&&+\mu ^{l}\sigma _{r+1,l}^{\infty }-\sigma _{r}^{\infty }\left( \mu
_{o},A\right) \mu ^{l}\mu _{o,l}  \nonumber \\
&&+g^{jm}\left( R^{\overline{M}}\right) ^{\perp }(\frac{\partial \psi }{%
\partial x_{i}},\mu )\frac{\partial \psi }{\partial x_{m}}\left(
T_{r}^{\infty }\right) _{j}^{i}  \TCItag{21}  \label{21} \\
&&+\lambda \left( \sigma _{1}^{\infty }\sigma _{r+1}^{\infty }-\left(
r+2\right) \sigma _{r+2}^{\infty }\right)   \nonumber \\
&&-\lambda \mu _{o}\sigma _{r+1}^{\infty }+\sigma _{r}^{\infty }(\mu _{o},A)%
\frac{\partial \mu _{o}}{\partial t}.  \nonumber
\end{eqnarray}%
The expression of $\frac{\partial dV}{\partial t}$is standard and it is
given by%
\[
\frac{\partial dV}{\partial t}=\left( -\lambda \sigma _{1}+\mu
_{,l}^{l}\right) dV
\]%
and by (\ref{2}), we get%
\begin{equation}
\frac{\partial dV}{\partial t}=\left( -\lambda \left( \sigma _{1}^{\infty
}-\mu _{o}\right) +\mu _{,l}^{l}\right) dV.  \tag{22}  \label{22}
\end{equation}%
Combining the expressions (\ref{21}) and (\ref{22}) we obtain the expression
of the integrand in Theorem \ref{thm1}

In the particular case where $\overline{M}$ is of constant curvature $c$, we
have%
\begin{eqnarray*}
g^{jm}R^{\overline{M}}(\nu ,\frac{\partial \psi }{\partial x_{m}},\frac{%
\partial \psi }{\partial x_{i}},X)\left( T_{r}^{\infty }\right) _{j}^{i}
&=&\lambda c\text{trace}\left( T_{r}^{\infty }\right) \\
&=&\lambda c\left( \left( n-r\right) \sigma _{r}^{\infty }+\mu _{o}\sigma
_{r-1}^{\infty }\right)
\end{eqnarray*}%
and in addition, we know that the covariant derivative of the second
fundamental form satisfies the relation%
\[
A_{ml},_{i}=A_{mi},l 
\]%
consequently:

\begin{corollary}
\label{cor1} If the ambient manifold $M$ is of constant curvature $c,$ we
have: for any $r\geq 2$ 
\begin{eqnarray}
\frac{d}{dt}\int_{M}\sigma _{r}^{\infty }dV &=&\int_{M}\left\{ \lambda
\left( -\left( r+1\right) \sigma _{r+1}^{\infty }+\mu _{o}\sigma
_{r}^{\infty }-\sigma _{r-1}^{\infty }\mu ^{l}\mu _{o,l}\right) +\lambda
,_{ij}\left( T_{r}^{\infty }\right) ^{ij}\right.  \nonumber \\
&&\left. +c\lambda \left( \left( n-r+1\right) \sigma _{r-1}^{\infty }+\mu
_{o}\sigma _{r-2}^{\infty }\right) +\sigma _{r}^{\infty }\frac{\partial \mu
_{o}}{\partial t}\right\} dV.  \TCItag{23}  \label{23}
\end{eqnarray}
\end{corollary}

\begin{remark}
In the particular case $\mu _{o}=0$, i.e. $\sigma _{r}^{\infty }=\sigma _{r}$
we recover the result in \cite{Reilly}.
\end{remark}

\section{Hypersurfaces in Euclidean space}

We restrict ourselves to the case $\mu _{o}$ is a constant

\begin{definition}
A hypersurface in an Euclidean space is said $\sigma _{r}^{\infty }$-minimal
if $\left( r+1\right) \sigma _{r+1}^{\infty }-\mu _{o}\sigma _{r}^{\infty }$
vanishes identically.
\end{definition}

As in the paper of Reilly (see \cite{Reilly}) we will express the minimality
of an hypersurface in terms of partial differential equations. Let $\psi
=(\psi _{1},...,\psi _{n+1})$ be the position vector of the hypersurface $M$
in the Euclidean space $E^{n+1}$ and $\psi _{,ij}=(\psi _{1,ij},...,\psi
_{n+1,ij})$ the second covariant derivative of $x$ on $M$. We have%
\begin{eqnarray*}
\psi _{,ij} &=&\frac{\partial ^{2}\psi }{\partial x_{i}\partial x_{j}}-d\psi
(\nabla _{\frac{\partial }{\partial x_{i}}}\frac{\partial }{\partial x_{j}})
\\
&=&\frac{\partial ^{2}\psi }{\partial x_{i}\partial x_{j}}-d\psi (\nabla _{%
\frac{\partial }{\partial x_{i}}}\frac{\partial }{\partial x_{j}}) \\
&=&\frac{\partial ^{2}\psi }{\partial x_{i}\partial x_{j}}-\nabla _{\psi
_{\ast }\frac{\partial }{\partial x_{i}}}\psi _{\ast }\frac{\partial }{%
\partial x_{j}} \\
&=&\left\langle \nabla _{\psi _{\ast }\frac{\partial }{\partial x_{i}}}\psi
_{\ast }\frac{\partial }{\partial x_{j}},\nu \right\rangle \nu 
\end{eqnarray*}%
where $\nabla $ is the covariant derivative on $M$ and $\nu $ denotes the
unit normal vector field to $M$ so 
\begin{eqnarray*}
\left( T_{r}^{\infty }\right) ^{ij}\psi _{,ij} &=&\left( T_{r}^{\infty
}\right) ^{ij}A_{ij}\nu =trace(AT_{r}^{\infty })\nu  \\
&=&(r+1)\sigma _{r+1}^{\infty }-\mu _{o}\sigma _{r}^{\infty }.
\end{eqnarray*}
Observe that in this case ($\mu _{o}=$constant ) the Newton tensor $\left(
T_{r}^{\infty }\right) ^{ij}$ is of free divergence so we have established

\begin{proposition}
In case $\mu _{o}=$constant , a hypersurface in an Euclidean space is $%
\sigma _{r}^{\infty }$-minimal if and only if each component of its position
vector field satisfies the partial differential equation $\left(
T_{r}^{\infty }\right) ^{ij}\psi _{,ij}=0$ or in its divergence form $\left(
\left( T_{r}^{\infty }\right) ^{ij}\psi _{,i}\right) _{j}=0$
\end{proposition}

Suppose $\mu _{o}=$constant and consider the function $\phi :M\rightarrow R$%
, given by $\phi (x)=\left\langle B(x),\nu \left( x\right) \right\rangle $
where $B\left( x\right) $ and $\nu \left( x\right) $ are respectively any
vector in $E^{n+1}$ and the normal one.%
\[
B=\phi \nu +\left\langle B,\frac{\partial \psi }{\partial x_{k}}%
\right\rangle \frac{\partial \psi }{\partial x_{k}} 
\]%
or briefly%
\[
B=\phi \nu +B_{k}\frac{\partial \psi }{\partial x_{k}} 
\]%
so

\begin{eqnarray*}
\phi _{,i} &=&\left\langle B,\nabla _{i}\nu \right\rangle  \\
&=&-B_{k}A_{ik}
\end{eqnarray*}%
and%
\begin{equation}
\phi _{,ij}=-\phi A_{ik}^{2}-B_{k}A_{ik,j}.  \tag{26}  \label{26}
\end{equation}%
Multiplying both sides of (\ref{26}) by $\left( T_{r}^{\infty }\right)
_{j}^{i},$and thanks to the Codazzi formula, proposition (\ref{prop3}) and
formula (\ref{9}) we obtain%
\begin{eqnarray*}
\phi _{,ij}\left( T_{r}^{\infty }\right) _{j}^{i} &=&-\phi trace\left(
A^{2}T_{r}^{\infty }\right) -B_{k}A_{ij,k}\left( T_{r}^{\infty }\right)
_{j}^{i} \\
&=&-\phi \left( \sigma _{1}^{\infty }\sigma _{r+1}^{\infty }-\left(
r+2\right) \sigma _{r+2}^{\infty }\right) -\left( r+1\right) B_{k}\sigma
_{r+1,k}^{\infty }+\mu _{o}B_{k}\sigma _{r,k}^{\infty }\text{.}
\end{eqnarray*}%
So, we established

\begin{proposition}
The $\sigma _{r+1}^{\infty }-$ mean curvature of a hypersurface in an
Euclidean space is constant if and only if the components of the normal
vector field satisfy the partial differential equation 
\[
\phi _{,ij}\left( T_{r}^{\infty }\right) _{j}^{i}=-\phi \left( \sigma
_{1}^{\infty }\sigma _{r+1}^{\infty }-\left( r+2\right) \sigma
_{r+2}^{\infty }\right) +\mu _{o}B_{k}\sigma _{r,k}^{\infty }\text{.} 
\]
\end{proposition}

\section{Hypersurfaces in the round unit sphere}

We consider hypersurfaces of the unit round sphere $S^{n}\subset E^{n+1}.$
From the Corollary \ref{cor1} the Euler-Lagrange equation is expressed by:%
\[
2\sigma _{2}^{\infty }-\mu _{o}\sigma _{1}^{\infty }-n=0 
\]%
for $r=1$, and 
\begin{equation}
\left( r+1\right) \sigma _{r+1}^{\infty }-\left( n-r+1\right) \sigma
_{r-1}^{\infty }-\mu _{o}\left( \sigma _{r}^{\infty }-\sigma _{r-2}^{\infty
}\right) =0  \tag{27}  \label{27}
\end{equation}%
for $r\geq 2.$

\begin{definition}
A hypersurface in the unit round sphere is said $\sigma _{r}^{\infty }$%
-minimal with $r\geq 2$ if 
\[
\left( r+1\right) \sigma _{r+1}^{\infty }-\left( n-r+1\right) \sigma
_{r-1}^{\infty }-\mu _{o}\left( \sigma _{r}^{\infty }-\sigma _{r-2}^{\infty
}\right) =0\text{.} 
\]
\end{definition}

Let $\psi =(\psi _{1},...,\psi _{n+2})$ be the position vector of the
hypersurface $M$ in the unit round sphere $S^{n+1}\subset E^{n+2}$ and $\psi
_{,ij}=(\psi _{1,ij},...,\psi _{n+2,ij})$ the second covariant derivative of 
$x$ on $M$. If $\nabla $ denotes the covariant derivative on $M$ induced by
the covariants derivative $\nabla ^{S^{n+1}}$ on the unit sphere. We have 
\begin{eqnarray}
\psi _{,ij} &=&\nabla _{\frac{\partial }{x_{j}}}\nabla _{\frac{\partial }{%
\partial x_{i}}}\psi =\nabla _{\frac{\partial }{x_{j}}}d\psi (\frac{\partial 
}{\partial x_{i}})  \nonumber \\
&=&\left\langle \nabla _{\frac{\partial }{x_{j}}}^{S^{n+1}}d\psi (\frac{%
\partial }{\partial x_{i}}),\nu \right\rangle \nu +\left\langle \nabla _{%
\frac{\partial }{x_{j}}}^{S^{n+1}}d\psi (\frac{\partial }{\partial x_{i}}%
),\psi \right\rangle \psi   \nonumber \\
&=&\left\langle \nabla _{\frac{\partial }{x_{j}}}^{S^{n+1}}d\psi (\frac{%
\partial }{\partial x_{i}}),\nu \right\rangle \nu -\left\langle \frac{%
\partial \psi }{\partial x_{i}}),\frac{\psi }{\partial x_{j}}\right\rangle
\psi   \TCItag{28}  \label{28} \\
&=&-g_{ij}\psi +A_{ij}\nu   \nonumber
\end{eqnarray}%
where $\nu $ is the unit normal vector field to $M.$ In order to
characterize the $\sigma _{r}^{\infty }$-minimality of the sub-manifolds of
the unit sphere, we multiply both sides of (\ref{28}) by $(r-1)\left(
T_{r}^{\infty }\right) ^{ij}-(n-r+1)\left( T_{r-2}^{\infty }\right) ^{ij}$
and sum to infer 
\[
\left( (r-1)\left( T_{r}^{\infty }\right) ^{ij}+(n-r+1)\left(
T_{r-2}^{\infty }\right) ^{ij}\right) \psi _{,ij}=
\]%
\[
-\left( (r-1)\left( \left( n-r\right) \sigma _{r}^{\infty }+\mu _{o}\sigma
_{r-1}^{\infty }\right) -(n-r+1)\left( \left( n-r+2\right) \sigma
_{r-2}^{\infty }+\mu _{o}\sigma _{r-3}^{\infty }\right) \right) \psi 
\]%
\[
\left( r-1\right) \left[ \left( \left( r+1\right) \sigma _{r+1}^{\infty
}-(n-r+1)\sigma _{r-1}^{\infty }\right) -\mu _{o}\left( \sigma _{r}^{\infty
}-\sigma _{r-2}^{\infty }\right) \right] \nu +n\mu _{o}\sigma _{r-1}^{\infty
}\nu 
\]%
Hence we proven:

\begin{proposition}
A hypersurface in the unit round sphere is $\sigma _{r}^{\infty }$-minimal $%
r\neq 1$ if \ and only if each component of the position vector field is
solution of the partial differential equation 
\[
\left( (r-1)\left( T_{r}^{\infty }\right) ^{ij}+(n-r+1)\left(
T_{r-2}^{\infty }\right) ^{ij}\right) \psi _{,ij}=n\mu _{o}\sigma
_{r-1}^{\infty }\nu 
\]%
\[
-\left( (r-1)\left( \left( n-r\right) \sigma _{r}^{\infty }+\mu _{o}\sigma
_{r-1}^{\infty }\right) -(n-r+1)\left( \left( n-r+2\right) \sigma
_{r-2}^{\infty }+\mu _{o}\sigma _{r-3}^{\infty }\right) \right) \psi . 
\]
\end{proposition}

\begin{remark}
In case $r=1$, a hypersurface in the unit round sphere is $\sigma
_{1}^{\infty }$-minimal if and only if $\ $%
\begin{equation}
2\sigma _{2}^{\infty }-\mu _{o}\sigma _{1}^{\infty }-n=0  \tag{29}
\label{29}
\end{equation}%
or equivalently 
\[
\psi _{,ij}T_{1}^{ij}=n\nu -\left( \left( n-1\right) \sigma _{1}^{\infty
}+\mu _{o}\right) \psi 
\]%
obtained by multiplying both sides of (\ref{28}) by $T_{1}^{\infty }$ and
summing.
\end{remark}


\begin{thebibliography}{9}
\bibitem{M.A1} M. Abdelmalek, M. Benalili, K. Niedzialomski, Geometric
Configuration of Riemannian submanifolds of arbitrary codimension, J. Geom
108 (2017), 803-823.

\bibitem{M.A2} M. Abdelmalek, M. Benalili, Some integral formulae on
weighted manifolds ( submitted )

\bibitem{Case} J. S. Case, A notion of the weighted $\sigma _{k}$-curvature
for manifolds with density, Adv. Math. 295 (2016), 150--194.

\bibitem{Cis} M. Ciska-Niedzia\l omska, Ma\l gorzata, K. Niedzia\l omski,
Rodin's formula in arbitrary codimension. Ann. Acad. Sci. Fenn. Math. 44
(2019), no. 1, 221--229.

\bibitem{K.N1} K. Niedzia\l omski, An integral formula for Riemannian
G-structures with applications to almost Hermitian and almost contact
structures. Ann. Global Anal. Geom. 56 (2019), no. 1, 167--192.

\bibitem{K.N2} K. Niedzia\l omski, Geometric structures on Riemannian and
Finsler manifolds---integral formulae, minimality, entropy. Folia Math. 20
(2018), no. 1, 3--16.

\bibitem{Reilly} R.C. Reilly, Variational properties of functions of the
mean curvatures for hypersurfaces in space forms. J. Differential Geometry 8
(1973) 465-477.

\bibitem{Reilly2} R.C. Reilly, Variational properties of mean curvatures,
Proc. Summer sem. Canad. Math. Congress, 1971, 102-114.

\bibitem{Rund} H. Rund, Invariant theory of variational problems on
subspaces of a Riemannian manifold. Hambourger Math. Einsenchriften no.5,
1971.
\end{thebibliography}
\end{document}